\documentclass[11pt]{amsart}
\usepackage{cite}
\usepackage{amsmath,amssymb}
\usepackage{amsthm}
\usepackage{enumerate}

\newtheorem{theorem}{Theorem}
\newtheorem{definition}{Definition}
\newtheorem{corollary}{Corollary}
\newtheorem{lemma}{Lemma}

\numberwithin{theorem}{section}

\begin{document}
\title{On equality of certain automorphism groups}
\author{Surjeet Kour, Vishakha}\thanks{}
\address{Discipline of Mathematics, Indian Institute of Technology, Gandhinagar  382424, India.}
\email{surjeetkour@iitgn.ac.in,  vishakha.maths@iitgn.ac.in} 

\begin{abstract}
Let $G = H\times A$ be a group, where $H$ is a purely non-abelian subgroup of $G$ and $A$ is a non-trivial abelian factor of $G$. Then, for $n \geq 2$, we show that there exists an isomorphism $\phi : Aut_{Z(G)}^{\gamma_{n}(G)}(G) \rightarrow Aut_{Z(H)}^{\gamma_{n}(H)}(H)$ such that $\phi(Aut_{c}^{n-1}(G))=Aut_{c}^{n-1}(H)$.
Also, for a finite non-abelian $p$-group $G$ satisfying a certain natural hypothesis, we give some necessary and sufficient conditions for $Autcent(G) = Aut_c^{n-1}(G)$.
Furthermore, for a finite non-abelian $p$-group $G$ we study the equality of $Autcent(G)$ with $Aut_{Z(G)}^{\gamma_{n}(G)}(G)$.

\end{abstract}

\maketitle

\section{{Introduction}}

Let $G$ be a group with the center $Z(G)$ and the automorphism group $Aut(G)$.
For an abelian group $H$,  $Hom(G, H)$ denotes the abelian group of all homomorphisms from $G$ to $H$.
For a given group $G$ along with some nice properties, it would be interesting to examine the properties of
its automorphism group and to determine relations among various subgroups of its automorphism group.
So it is  natural  to look for appropriate conditions on a group $G$ such that one can describe its automorphism group
$``Aut(G)"$ or could find out certain subgroups of $Aut(G)$ are isomorphic.
Also, the group of homomorphisms is far better understood than the group of automorphisms and therefore it would be in a great interest of many mathematicians
to see isomorphisms between subgroups of an automorphism group and suitable homomorphism groups. Recently, many mathematicians got interested to study the equality of certain subgroups of an automorphism group and a number of results have been proved in this direction.
Many interesting results have been proved in \cite{DH}, \cite{Manoj}, \cite{Kalra}, \cite{Yadav} etc..

Most of the results, which have been proved in this direction, are either for a finite group or for a finite non-abelian $p$-group of nilpotency
class at most two. Our main motivation behind this work is to study relations among certain subgroups of the automorphism group of two groups (need not be finite)
and to study the relations among certain subgroups of the automorphism group of a finite non-abelian $p$ group.  In Section \ref{s3} and \ref{s4}, we obtain quite a few interesting results for infinite groups and for finite non-abelian $p$-groups. These results are generalizations of some results which were proved
either for a finite group or for a finite non-abelian $p$-group of class at most two.

An automorphism $\alpha$ of $G$ is called central if $g^{-1}\alpha(g) \in Z(G)$ for all $g\in G$.
The set of all central automorphisms, denoted by $Autcent(G)$, is a normal subgroup of $Aut(G)$.
 In \cite{Yen}, Adney and Yen have obtained a very
interesting relation between the central automorphism group $Autcent(G)$ of a finite non-abelian group G and the homomorphism
group Hom(G, Z(G)). They proved that, for a purely non-abelian finite group G (a finite non-abelian group $G$ is said to be purely non-abelian, if it has no non-trivial abelian direct factor), there exists a
bijection between Autcent(G) and Hom(G, Z(G)). This result is being used extensively in the study of central automorphism group
and it is also referred at various places in this article.


Let $M$ and $N$ be two normal subgroups of $G$. Let $Aut^{M}(G)$ denote the group of all automorphisms of $G$ which fix $M$ setwise and act trivially on $G/M$ and let $Aut_{N}(G)$ denote the group of all automorphisms of $G$ which fix  $N$ element wise. The group $Aut^{M}(G)\cap Aut_{N}(G)$ is denoted by $Aut_{N}^{M}(G)$. Note that $ Autcent(G) =  Aut_{\gamma_2(G)}^{Z(G)}(G)$. %

\begin{definition}
An automorphism $\beta$ of $G$ is called $n^{th}$ class preserving if for each $g\in G$, there exists $x\in \gamma_{n}(G)$ such that $\beta(g)= x^{-1}gx$,
where $\gamma_n(G)$ denotes the $n^{th}$ term of lower central series.
We denote the group of all $n^{th}$ class preserving automorphisms by $Aut_{c}^{n}(G)$.
\end{definition}
Note that $Aut_c^1(G)$ is denoted by $Aut_c(G)$ and called the group of all class preserving automorphisms.

Let $H_1, H_2, \ldots, H_k$ be $k$ groups and let $G$ be the direct product of $H_i's$.
In Section \ref{s3}, we obtain relations between subgroups of $Aut(G)$ and subgroups of $Aut(H_j)$, which fix certain subgroups and quotient groups of the group $G$ and the group $H_j$. More precisely, in Theorem \ref{t2} we prove that for two normal subgroups $M_j$ and  $N_j$ of $H_j$, if $M=\{1\} \times \cdots \times M_j \times \cdots \times \{1\}$ and $N  =H_1 \times \cdots \times N_j \times \cdots \times H_k$, then $Aut_{N}^{M}(G)$ is isomorphic to $Aut_{N_j}^{M_j}(H_j)$.

As a consequence of Theorem \ref{t2}, we  obtain Theorem \ref{c1} which states that if $ G= H \times A$ is  a  non-abelian group, where $H$ is a purely non-abelian subgroup of $G$ and $A$ is a non-trivial abelian factor of $G$, then there exists an isomorphism
$\phi : Aut_{Z(G)}^{\gamma_{n}(G)}(G) \rightarrow Aut_{Z(H)}^{\gamma_{n}(H)}(H)$ such that $\phi(Aut_{c}^{n-1}(G))=Aut_{c}^{n-1}(H)$.

In \cite[Theorem 4.1]{Manoj}, Yadav proved that if $G$ and $H$ are two finite non-abelian isoclinic
groups, then $Aut_c(G)$ is isomorphic to $Aut_c(H)$. Later on,  Rai in \cite[Theorem A]{Rai} has extended Yadav's
above result to the group $Aut_{Z(G)}^{\gamma_2(G)}(G)$, which states that if two finite non-abelian groups $G$ and $H$ are isoclinic,
then there exists an isomorphism $ \phi : Aut_{Z(G)}^{\gamma_2(G)}(G)\rightarrow Aut_{Z(H)}^{\gamma_2(H)}(H)$ such that
$ \phi(Aut_c(G))= Aut_c(H)$.

When $G = H \times A$, Theorem A of Rai \cite{Rai} , the extension of Yadav \cite[Theorem 4.1]{Manoj},  can be obtained from  Theorem \ref{c1} by
taking $G$ to be a finite group and $n=2$.

In the same article \cite{Manoj}, Yadav  also proved that for a finite nilpotent group $G$ of class two, the group of all class preserving automorphisms
is isomorphic to a subgroup of $ Hom(G/Z(G) , \gamma_2(G))$.
In Section \ref{s3}, Theorem \ref{t3}, we generalize this result and prove that if $G$ is a finite nilpotent group of class at most $n$ and $H$ is a subgroup of
$G$ such that $ \gamma_n(G) \subseteq H \subseteq Z(G)$, then the group of $(n-1)^{th}$ class preserving automorphism is isomorphic
to a subgroup of $ Hom(G/H , \gamma_n(G))$. Yadav's above result can be obtained from Theorem \ref{t3} by taking $n=2$ and $ H= Z(G)$.

In Section \ref{s4}, we  obtain various equalities between the group of central automorphisms of a finite non-abelian $p$-group $G$ and suitable subgroups of the
automorphism group $Aut(G)$ by putting some conditions on the group $G$.

In \cite[Theorem 3.1]{Kalra}, Kalra and Gumber proved that for a finite $p$-group $G$, $ Autcent(G) = Aut_c(G)$ if and only if $Aut_{c}(G)$ is isomorphic to $Hom(G/Z(G),\gamma_{2}(G))$ and $\gamma_{2}(G) = Z(G)$.
In Theorem \ref{t4}, we prove that  for a finite non-abelian $p$-group $G$ and for $n \geq 2$, if $Autcent(G) =Aut_{c}^{n-1}(G) $ then $\gamma_{n}(G)\leq Z(G)$ and $d(Z(G)) = d (\gamma_{n}(G))$ also if $exp(G/\gamma_2(G))>var(\gamma_{n}(G),Z(G))$ then $\gamma_{n}(G)=Z(G)$ and $Aut_{c}^{n-1}(G) $ is isomorphic to $Hom(G/Z(G), \gamma_n(G))$. We obtain Kalra and Gumber \cite[Theorem 3.1]{Kalra} as  Corollary \ref{c3} of Theorem \ref{t4}.

In \cite[Theorem B(1)]{Rai},  Rai also studied the equality of the group of central automorphisms with $Aut_{Z(G)}^{\gamma_{2}(G)}(G)$. He proved that for a finite non-abelian $p$-group $G$ of nilpotency class two, $Autcent(G) = Aut_{Z(G)}^{\gamma_{2}(G)}(G)$ if and only if $\gamma_{2}(G) = Z(G)$.
In Theorem \ref{t5}, for $n\geq 2$,  we prove that for a finite non-abelian $p$-group $G$, if $ Autcent(G) = Aut_{Z(G)}^{\gamma_{n}(G)}(G)$ then
$\gamma_{n}(G)\leq Z(G)$ and $d(Z(G)) = d (\gamma_{n}(G))$ also if $exp(G/\gamma_2(G))> var(\gamma_{n}(G),Z(G))$ then $\gamma_{n}(G) = Z(G)$, which generalizes Rai's above result.

%

\section{{Notations and Preliminaries}}\label{s2}

Throughout the article, $n$ denotes a positive integer and $p$ denotes a prime number. We use $1$ to denote the identity of a group. For basic definitions we refer \cite{Dummit}. Here we recall some definitions:

The exponent of a group $G$, denoted by $exp(G)$, is the smallest positive integer $n$ such that $g^{n} = 1$ for all $g \in G$.
For a finite abelian group $G$,  $d(G)$ denotes the rank of $G$. Let $G$ and $H$ be two finite abelian $p$-groups. Let
$$G =  C_{p^{n_1}}\times C_{p^{n_2}}\times \cdots \times C_{p^{n_s}},$$  $$H = C_{p^{m_1}}\times C_{p^{m_2}}\times \cdots \times C_{p^{m_t}}$$ be the cyclic decompositions of $G$ and $H$ respectively, where $n_i \geq n_{i+1}$ and $m_j \geq m_{j+1}$ are positive integers and $C_{p^{i}}$ denotes the cyclic group of order $p^i$. If $d(G) = d(H)$, that is $s =t$, and $G$ is a proper subgroup of $H$, then there exists a  unique integer $r$ such that $n_r < m_r$ and $ n_j= m_j$ for all $ r+1 \leq j \leq t$.  In this case, we define $var(G,H)= p^{n_r}$. In other words, $var(G,H)$ denotes the order of the last cyclic factor of $G$ whose order is less than that of the corresponding cyclic factor of $H$. If $ G = H$, then we say $var(G,H) =1$.



The following lemmas are  well known.

\begin{lemma}\label{l3}
Let $G$ be a nilpotent group of class $n$. Then $exp(G/Z (G))$ divides $exp(\gamma_n(G))$. Furthermore, if $n=2$, then $exp(G/Z(G))= exp(\gamma_2(G))$.
\end{lemma}

\begin{lemma}\label{l4}
Let $G$ be a group and let  $H$, $K$ be two normal subgroups of $G$ such that $H\subseteq K$. Then $exp(G/K)$ divides $exp(G/H)$.
\end{lemma}

%
\begin{lemma}\label{l5}
Let $G=H\times K$, where $H$ and $K$ are two groups. If $K$ is abelian, then $G$ is nilpotent  of class $n$ if and only if $H$ is nilpotent of class $n$.
\end{lemma}

For a purely non-abelian finite group $G$,  Adney and Yen proved the following result.
\begin{theorem}\cite[Theorem 1]{Yen}\label{t1}
Let $G$ be a purely non-abelian finite group. Then $|Autcent(G)|$ $=|Hom(G,Z(G))|$.
\end{theorem}
Note that $Hom(G,Z(G))$ is isomorphic to $Hom(G/\gamma_i(G),Z(G))$ for all $i\geq 2$. So, for $ i \geq 2$ , we have $$|Autcent(G)|=|Hom(G/\gamma_i(G),Z(G))|.$$

The following lemmas are proved in  \cite{Bid} and \cite{DH}.

\begin{lemma}\cite[Corollary 3.3]{Bid}\label{l6}
Let $G = H \times K$, where $H$ and $K$ have no common direct factor. Then $$|Autcent(G)|=|Autcent(H)||Autcent(K)||Hom(H, Z(K))||Hom (K, Z(H ))|.$$
\end{lemma}

\begin{lemma}\cite[Lemma 2.2(i)]{DH}\label{l1}
Let $G, H$ and $K$ be finite abelian $p$-groups with $H \subseteq K$. Then the following statements are equivalent:
\begin{enumerate}[$(1)$]
\item $Hom(G, H) = Hom(G, K)$;

\item Either $H = K$ or $d(H) = d(K)$ and $exp(G) \leq var(H, K)$.
\end{enumerate}
\end{lemma}

\section{{Class Preserving Automorphisms}}\label{s3}

 Let $H_1, H_2, \ldots, H_k$ be $k$ groups (need not be finite) and let $G$ be the direct product of $H_i's$. Let $\pi_j$ denote the projection of $G$ on $H_j$.
 Here we prove some relations among  the subgroups of $Aut(G)$ and $Aut(H_j)$, where $j \in \{1, \ldots, k\}$.

\begin{theorem}\label{t2}
Let $H_1, H_2, \ldots, H_k$ be $k$ groups and let $G = H_1 \times H_2 \times \cdots \times H_k$.  For  a  fixed  $j\in \{1,\ldots ,k\}$,  let $M_j$ and  $N_j$ be two normal subgroups of $H_j$. If  $M=\{1\} \times \cdots \times M_j \times \cdots \times \{1\}$ and $N  =H_1 \times \cdots \times N_j \times \cdots \times H_k$, then

\begin{enumerate}[$(1)$]
\item For each $ f \in Aut_{N}^{M}(G)$ the map $\alpha_f : H_j \rightarrow H_j$ given by $\alpha_{f}(h_j) = \pi_j o f(1,\ldots,h_j,\ldots,1)$
is an automorphism and it belongs to $Aut_{N_j}^{M_j}(H_j)$.
\item The map $\phi :Aut_{N}^{M}(G) \rightarrow  Aut_{N_j}^{M_j}(H_j)$ given by $\phi(f) = \alpha_{f}$ is an isomorphism.
\end{enumerate}
\end{theorem}

%

\begin{proof}
\begin{enumerate}
\item First we show that $\alpha_f$ is an automorphism. Clearly, $\alpha_{f}$ is a homomorphism as $f$ and $\pi_j$ are homomorphisms.

To prove $\alpha_{f}$ is injective, consider $h_j \in H_j$ such that $\alpha_{f}(h_j)=1$. Note that $(1,\ldots,h_j,\ldots,1) \in G$.  Since $f\in Aut_{N}^{M}(G)$, therefore there exists $m_j \in M_j$ such that $f(1,\ldots,h_j,\ldots,1)=(1,\ldots,h_j m_j,\ldots,1)$. Then, $\alpha_{f}(h_j)= \pi_j(1,\ldots,h_j m_j,\ldots,1)$ = $ h_jm_j=1$. Thus $f(1,\ldots,h_j,\ldots,1)=(1,\ldots,1)$. Since $f$ is injective, therefore  $h_j=1$. Hence $\alpha_{f}$ is injective.

Now we show that $\alpha_{f}$ is surjective. Let $h_j\in H_j$.  Since $f$ is onto, therefore there exists  $(a_1,\ldots,a_k)\in G$ such that $f(a_1,\ldots,a_k)=(1,\ldots,h_j,\ldots,1)$. As $f$ fixes $N$ element wise, $f(1,\ldots,a_j,\ldots,1)=(a_1^{-1},\ldots,h_j,\ldots,a_k^{-1})$.
Therefore,  $\alpha_f(a_j)= \pi_j (a_1^{-1},\ldots,h_j,\ldots,a_k^{-1})=h_j$. Hence $\alpha_f$ is an automorphism.
Also $h_j^{-1}\alpha_f(h_j)\in M_j$ for all $h_j\in H_j$ and $\alpha_f(n_j)=n_j$ for all $n_j\in N_j$.
Thus $\alpha_f \in Aut_{N_j}^{M_j}(H_j)$.\\

\item  To prove that $\phi$ is a homomorphism, consider $f_1, f_2 \in Aut_{N}^{M}(G)$
and $h_j \in H_j$. Let  $x_1,x_2 \in M_j$ be such that
$f_2(1,\ldots,h_j,\ldots,1)=(1,\ldots,h_jx_2,\ldots,1)$ and
$f_1 o f_2(1,\ldots,h_j,\ldots,1)= f_1(1,\ldots,h_jx_2,\ldots,1)=(1,\ldots,h_j x_2x_1,\ldots,1)$.

Then
\begin{equation*}
\begin{split}
\alpha_{f_1}o\alpha_{f_2}(h_j) &= \alpha_{f_1}(\pi_j o f_2(1,\ldots,h_j,\ldots,1)) = \alpha_{f_1}(h_jx_2)\\
                                      & = \pi_j o f_1(1,\ldots,h_jx_2,\ldots,1) = h_j x_2x_1\\
 \Rightarrow \alpha_{f_1}o\alpha_{f_2}(h_j)                                    & = \pi_j o (f_1 of_2)(1,\ldots,h_j,\ldots,1)
                                     = \alpha_{f_1of_2}(h_j).
\end{split}
\end{equation*}

Therefore, $ \phi(f_1 o f_2) = \alpha_{f_1of_2}= \alpha_{f_1}o\alpha_{f_2} = \phi(f_1)o \phi(f_2)$.

To prove that $\phi$ is injective, consider $f \in Aut_{N}^{M}(G)$ such that  $\phi(f)= 1$, that is,
$\phi(f)(h_j)= \alpha_f(h_j)=h_j$  for all $h_j \in H_j$.  Since $f \in Aut_{N}^{M}(G)$, therefore there exists $m_j \in M_j$ such that $f(1,\ldots,h_j,\ldots,1)=(1,\ldots,h_jm_j,\ldots,1)$. Now  $ h_j = \alpha_f(h_j)= \pi_j o f(1, \ldots, h_j, \ldots, 1) = h_jm_j$ and therefore $  m_j= 1 $.
Thus $ f(1,\ldots,h_j,\ldots,1)=(1,\ldots,h_j,\ldots,1)$ for all $h_j\in H_j$.
As $f$ fixes $N$ element wise,  $ f(a_1,\ldots,a_k)= (a_1,\ldots,a_k)$  for all $(a_1,\ldots,a_k)\in G$. Hence $\phi$ is injective.

Now we show that $\phi$ is surjective.
Let $\psi \in Aut_{N_j}^{M_j}(H_j)$. Define $f_{\psi}: G \rightarrow G$ as $f_{\psi}(h_1,\ldots, h_j,\ldots,h_k)= (h_1,\ldots,\psi (h_j),\ldots,h_k)$.
It is easy to observe that $f_{\psi}\in Aut_{N}^{M}(G)$. Furthermore, $\phi(f_{\psi})= \psi$.
Thus $\phi$ is an isomorphism.
\end{enumerate}
\end{proof}

In \cite[Theorem A]{Rai}, Rai proved that if $G$ and $H$ are two finite isoclinic groups,  then there exists an isomorphism
$\Psi : Aut_{Z(G)}^{\gamma_{2}(G)}(G) \rightarrow Aut_{Z(H)}^{\gamma_{2}(H)}(H)$ such that $\Psi (Aut_{c}(G))=Aut_{c}(H)$.
Note that if $G = H\times A$, where $H$ is a purely non-abelian subgroup of $G$ and $A$ is a non-trivial abelian factor of $G$,
then $G$ and $H$ are isoclinic.
As a consequence of Theorem \ref{t2},  we obtain the following result which generalizes the above result of Rai \cite[Theorem A]{Rai}
to any group (need not be finite) in the case $ G = H \times A$.

\begin{theorem}\label{c1}
Let $G = H\times A$ be a group, where $H$ is a purely non-abelian subgroup of $G$ and $A$ is a non-trivial abelian factor of $G$. Then, for $n\geq 2$, there exists an isomorphism $\phi : Aut_{Z(G)}^{\gamma_{n}(G)}(G) \rightarrow Aut_{Z(H)}^{\gamma_{n}(H)}(H)$ such that $\phi(Aut_{c}^{n-1}(G))=Aut_{c}^{n-1}(H)$.
\end{theorem}

\begin{proof}
Note that $Z(G)=Z(H)\times A$ and $\gamma_n(G)=\gamma_n(H) \times \{1\}$. Then, by Theorem \ref{t2}(2), we have an isomorphism $\phi : Aut_{Z(G)}^{\gamma_{n}(G)}(G) \rightarrow Aut_{Z(H)}^{\gamma_{n}(H)}(H)$ given by $\phi(f) = \alpha_f$, where $\alpha_f$ is the map defined in Theorem \ref{t2}(1). Next, we show that $\phi(Aut_{c}^{n-1}(G))=Aut_{c}^{n-1}(H)$. Let $\pi_H$ be the projection of $G$ on $H$.
Consider $f\in Aut_{c}^{n-1}(G) \subseteq Aut_{Z(G)}^{\gamma_{n}(G)}(G)$, then $\phi(f)  \in Aut_{Z(H)}^{\gamma_{n}(H)}(H)$.
Also for any $h\in H$, we have $\phi(f)(h)=\alpha_f(h)=\pi_H(f(h,1))= k^{-1}hk$; where $f(h,1)=(k^{-1}hk,1)$ for some $k\in \gamma_{n-1}(H)$. Thus $\phi(f) \in Aut_{c}^{n-1}(H)$. Hence $\phi(Aut_{c}^{n-1}(G)) \subseteq Aut_{c}^{n-1}(H)$.

 Now,  consider $\psi \in Aut_{c}^{n-1}(H)$ and define $f_{\psi}: G \rightarrow G$ as $f_{\psi}(h,a)=(\psi(h),a)$ for all $(h,a)\in G$. It is easy to see that $f_{\psi}\in Aut_c^{n-1}(G)$. Also for $h\in H,~\phi(f_{\psi})(h)= \alpha_{f_{\psi}}(h)= \pi_H(f_{\psi}(h,1))=\pi_H(\psi(h),a)=\psi(h)$. Hence $\phi(f_{\psi})=\psi$. Thus $\psi \in \phi(Aut_{c}^{n-1}(G))$. Hence $Aut_{c}^{n-1}(H) \subseteq \phi(Aut_{c}^{n-1}(G))$.
\end{proof}

\begin{lemma}\label{l7}
Let $G$ be a group. Let $M$ and $N$ be two normal subgroups of $G$ such that $M$ is abelian and $[N,M] = 1$. Then the following statements are true:
\begin{enumerate}[$(1)$]
\item{} For each $ f \in Aut_N^M(G)$ the map $\alpha_f : G/N \rightarrow M $ given by $\alpha_f(gN) = g^{-1}f(g)$ is well
defined.
\item{} If $f_1, f_2 \in Aut_N^M(G)$ and $f_1 \neq f_2$ then $\alpha_{f_1} \neq \alpha_{f_2}$.
\end{enumerate}
\end{lemma}

\begin{proof}
  \begin{enumerate}
  \item Let $g_1, g_2 \in G$ with $g_1 N = g_2 N$. Then $g_1 = g_2 n$ for some $n\in N$. Observe that
    $$g_1^{-1} f(g_1)= (g_2 n)^{-1} f(g_2 n) = n^{-1} g_2^{-1} f(g_2) f(n) =  n^{-1} g_2^{-1} f(g_2) n.$$
      Since $g_2^{-1} f(g_2) \in M$ and $[N, M] = 1$, therefore $g_2^{-1} f(g_2)$ commutes with $n$. Thus $g_1^{-1} f(g_1) = g_2^{-1} f(g_2)$.\\

   \item  If $f_1, f_2 \in Aut_N^M(G)$ and $f_1 \neq f_2$. Then $f_1(g) \neq f_2(g)$ for some $g\in G$.
  This implies that $g^{-1} f_1(g) \neq g^{-1} f_2(g)$. Therefore  $\alpha_{f_1} \neq \alpha_{f_2}$.
  \end{enumerate}
\end{proof}

A subgroup $H$ of a group $G$ is called central if $H \subseteq Z(G)$.

\begin{theorem}\label{t6}
Let $G$ be a group and let $M$ and $N$ be two normal subgroups of $G$. Suppose, $M$ is a central subgroup of $G$. Then the following statements are true:
\begin{enumerate}[$(1)$]
\item For each $ f \in Aut_N^M(G)$ the map $\alpha_f : G/N \rightarrow M $ given by $\alpha_f(gN) = g^{-1}f(g)$ is well
defined and it is a homomorphism.

\item If $G$ is finite and $M \subseteq N$, then the map $\phi : Aut_N^M(G) \rightarrow Hom(G/N , M)$ defined by
$\phi(f) =  \alpha_f $ is an isomorphism.
\end{enumerate}
\end{theorem}

\begin{proof}
  \begin{enumerate}
\item{} By Lemma \ref{l7}(1) the map $\alpha_f$ is well defined. Let $g_1, g_2 \in G$. Then
      $$\alpha_f(g_1N g_2N) = \alpha_f(g_1g_2N) =  g_2^{-1} g_1^{-1} f(g_1) f(g_2).$$
    Since $g_1^{-1} f(g_1) \in M \subseteq Z(G)$, therefore $\alpha_f(g_1N g_2N) = g_1^{-1} f(g_1) g_2^{-1} f(g_2) = \alpha_f(g_1N) \alpha_f(g_2N) $.
    Thus $\alpha_f$ is a homomorphism.\\

\item{} Let $f_1, f_2 \in Aut_N^M(G)$. We claim that $\alpha_{f_1} \alpha_{f_2} = \alpha_{f_1 o f_2}$. Let $g\in G$. Since $g^{-1} f_2(g) \in M \subseteq N$,
    therefore $gN = f_2(g)N$. Then
    \begin{equation*}
    \begin{split}
    (\alpha_{f_1} \alpha_{f_2}) (gN) &= \alpha_{f_1}(gN) \alpha_{f_2}(gN)\\
                                     &= \alpha_{f_1}(f_2(g) N) \alpha_{f_2}(gN)\\
                                     &= f_2(g^{-1})  f_1(f_2(g)) g^{-1} f_2(g).
    \end{split}
    \end{equation*}

    Since $g^{-1} f_2(g) \in M \subseteq Z(G)$, therefore $ g^{-1} f_2(g) f_2(g^{-1})  f_1(f_2(g))
     = g^{-1} f_1(f_2(g))$ and hence $(\alpha_{f_1} \alpha_{f_2}) (gN)  = \alpha_{f_1 o f_2}(gN)$.
    This proves that the map $\phi$ is a homomorphism.

    By Lemma \ref{l7}(2) $\phi$ is injective. It remains to show that $\phi$ is surjective. Let $\psi \in Hom(G/N, M)$ be arbitrary. We define the map $f_{\psi} : G \rightarrow G$ as $f_{\psi}(g) = g \psi(gN)$ for all $g\in G$.

    First we claim that $f_{\psi}$ is a homomorphism. Let $g_1, g_2 \in G$. Then
      $$f_{\psi}(g_1 g_2) = g_1 g_2 \psi(g_1 g_2 N) =  g_1 g_2 \psi(g_1 N)\psi(g_2 N).$$
    Since $\psi(g_1 N) \in M \subseteq Z(G)$, therefore $ g_1 g_2 \psi(g_1 N)\psi(g_2 N) = g_1 \psi(g_1 N) g_2 \psi(g_2 N)$ and hence $f_{\psi}(g_1 g_2) = f_{\psi}(g_1) f_{\psi}(g_2)$.

    Next we claim that $f_{\psi}$ is injective. Let $g\in G$ with $f_{\psi}(g) = 1$. Then $g \psi(gN) = 1$ and so $\psi(gN) = g^{-1}$.
    Since $\psi(gN) \in M \subseteq N$, it follows that $g^{-1} \in N$ and so $\psi(gN) = 1$. Thus $g =1$.\\
    Since the map $f_{\psi}$ is injective and $G$ is finite, therefore $f_{\psi}$ is surjective. Thus $f_{\psi} \in Aut(G)$.
    Now it is easy to observe that $f_{\psi} \in Aut_N^M(G)$ and $\phi(f_{\psi}) = \psi$.
    Hence the map $\phi$ is surjective.
  \end{enumerate}
\end{proof}

Let $G$ be a nilpotent group of class at most $n$ and $H$ be a central subgroup of $G$. We denote the group  $\{f\in Hom(G/H,\gamma_n(G))~|~~ f(gH)\in [g,\gamma_{n-1}(G)]$  $ \forall g\in G\}$ by $Hom_{c}(G/H, \gamma_n(G))$.

In \cite[Proposition 3.1]{Manoj}, Yadav proved that  for a finite nilpotent group $G$ of class 2,  $Aut_{c}(G)$ is isomorphic to $Hom_{c}(G/Z(G), \gamma_2(G))$. We generalize this result to a finite nilpotent group of class $n$. Also we observe that the result is true for any central subgroup $H$ of $G$ which contains $\gamma_n(G)$.

\begin{theorem}\label{t3}
Let $G$ be a finite nilpotent group of class at most $n$ and let $H$ be a central subgroup of $G$ such that $\gamma_n(G)\leq H$. Then, for  $n\geq 2$, $Aut_{c}^{n-1}(G)$ is isomorphic to $Hom_{c}(G/H, \gamma_n(G))$.
\end{theorem}

\begin{proof}
Since $\gamma_n(G) \leq  H$, therefore by Theorem \ref{t6}(2) if we take  $M = \gamma_n(G)$ and $N = H$, the map $\phi : Aut_H^{\gamma_n(G)}(G) \rightarrow Hom(G/H , \gamma_n(G))$ defined by $\phi(f) =  \alpha_f $ is an isomorphism.

Note that $Aut_{c}^{n-1}(G) \subseteq Aut_H^{\gamma_n(G)}(G)$ and hence the restriction map $\phi : Aut_{c}^{n-1}(G) \rightarrow Hom(G/H , \gamma_n(G))$ is a injective homomorphism. It remains to show that $\phi(Aut_{c}^{n-1}(G)) = Hom_{c}(G/H, \gamma_n(G))$. Let $f \in Aut_{c}^{n-1}(G)$ be arbitrary.  Then for a given $g\in G$, there exists $x\in \gamma_{n-1}(G)$ depending on $g$, such that $f(g)=x^{-1}gx$. Thus
$$\phi(f)(gH)=\alpha_f(gH)=g^{-1}f(g)= g^{-1}x^{-1}gx = [g,x] \in [g,\gamma_{n-1}(G)].$$
Therefore $\phi(f) \in Hom_{c}(G/H, \gamma_n(G))$ and hence $\phi(Aut_{c}^{n-1}(G)) \subseteq Hom_{c}(G/H, \gamma_n(G))$.\\

Let $\psi \in Hom_{c}(G/H, \gamma_n(G))$ be arbitrary. We define the map $f_{\psi} : G \rightarrow G$ by $f_{\psi}(g) = g\psi(gH)$ for all $g\in G$. Then by using the argument similar to Theorem \ref{t6}(2), $f_{\psi} \in Aut_H^{\gamma_n(G)}(G)$. Now it is easy to see that $f_{\psi} \in Aut_{c}^{n-1}(G)$ and $\phi(f_{\psi})
= \alpha_{f_{\psi}} = \psi$. Thus $Hom_{c}(G/H, \gamma_n(G)) \subseteq \phi(Aut_{c}^{n-1}(G))$.
Hence the map $\phi$ is an isomorphism from $Aut_{c}^{n-1}(G)$ to $Hom_{c}(G/H, \gamma_n(G))$.
\end{proof}

\begin{corollary}\label{c2}
Let $G$ be a finite nilpotent group of class $n$. Then $Aut_{c}^{n-1}(G)$ is isomorphic to $Hom_{c}(G/Z(G), \gamma_n(G))$.
\end{corollary}

\begin{proof}
Take $ H = Z(G)$. Then the  result follows from  Theorem \ref{t3}.
\end{proof}

Now  the result of Yadav \cite[Proposition 3.1]{Manoj} follows from  Corollary \ref{c2} by putting $n=2$.


\section{{p-groups with $Autcent(G)$  equal to either $Aut_{c}^{n-1}(G)$ or $Aut^{\gamma_n(G)}_{Z(G)}(G)$}}\label{s4}

Here we discuss the equality of central automorphisms  with $(n-1)^{th}$ class preserving automorphisms and with  $Aut^{\gamma_n(G)}_{Z(G)}(G)$.

\begin{theorem}\label{t4}
Let $G$ be a finite non-abelian $p$-group and  $n\geq 2$. Let $Aut_{c}^{n-1}(G)=Autcent(G)$. Then

\begin{enumerate}[$(1)$]
\item $\gamma_n(G) \leq Z(G)$ and $d(\gamma_n(G)) = d(Z(G))$.

\item If $exp(G/\gamma_2(G))>var(\gamma_{n}(G),Z(G))$, then $\gamma_{n}(G)=Z(G)$ and $Aut_{c}^{n-1}(G)$ is isomorphic to  $Hom(G/Z(G), \gamma_n(G))$.
\end{enumerate}
Conversely, if $\gamma_{n}(G)=Z(G)$ and $Aut_{c}^{n-1}(G)$ is isomorphic to  $Hom(G/Z(G), \gamma_n(G))$, then $Aut_{c}^{n-1}(G)=Autcent(G)$ and $exp(G/\gamma_2(G))>var(\gamma_{n}(G),Z(G))$.

\end{theorem}

%

\begin{proof}
\begin{enumerate}
\item Let $Aut_{c}^{n-1}(G)=Autcent(G)$. We first prove that $\gamma_{n}(G)\leq Z(G)$.
Let $g\in G$ and $x\in \gamma_{n-1}(G)$. Then $[g,x] = g^{-1}\sigma_{x^{-1}}(g)$, where $\sigma_{x^{-1}}$ is defined as $\sigma_{x^{-1}}(a)= x^{-1}ax$ for all $a\in G$. As $\sigma_{x^{-1}} \in Aut_{c}^{n-1}(G)= Autcent(G)$, we get
$  [g,x] = g^{-1}\sigma_{x^{-1}}(g) \in Z(G)$, and therefore $\gamma_n(G) \leq Z(G)$.

 To prove that $d(\gamma_n(G)) = d(Z(G))$,  first we show that $G$ is purely non-abelian. Suppose, if possible, $G$ has a non-trivial abelian direct factor, say $G=H\times A$ where $H$ is a purely non-abelian subgroup and $A$ is a non-trivial abelian direct factor of $G$. Note that $\gamma_n(H) \leq Z(H)$. Then using Lemma \ref{l6} and Theorem \ref{c1}, $|Autcent(G)|>|Autcent(H)|\geq |Aut_{c}^{n-1}(H)|=|Aut_{c}^{n-1}(G)|=|Autcent(G)|$, which is absurd. Thus $G$ is purely non-abelian.
Hence,  by Theorem \ref{t1},  $$|Autcent(G)| = |Hom(G/\gamma_{2}(G), Z(G))|.$$
Now for $f \in Aut_{c}^{n-1}(G)$,  by Theorem \ref{t6}(1), if we take $M = \gamma_n(G)$ and $N=\{1\}$, the map $\alpha_f \in Hom(G, \gamma_{n}(G))$. As $ \gamma_2(G) \subseteq ker(\alpha_f)$; $~~\alpha_f$ can be extended to a homomorphism from $G/\gamma_2(G)$ to $\gamma_{n}(G)$. Therefore, $|Aut_{c}^{n-1}(G)|\leq |Hom(G/\gamma_2(G),\gamma_{n}(G))|$. Thus by the given condition, $$Hom(G/\gamma_2(G),Z(G)) = Hom(G/\gamma_2(G),\gamma_{n}(G)).$$
Then by using Lemma \ref{l1}, $d(\gamma_n(G)) = d(Z(G))$.\\

\item Since $exp(G/\gamma_2(G))>var(\gamma_n(G),Z(G))$, therefore by using Lemma \ref{l1}, we get $\gamma_{n}(G)=Z(G)$.

Note that $Autcent(G)= Aut_{\gamma_n(G)}^{Z(G)}(G)$. Then by Theorem \ref{t6}(2), if we take $M = Z(G)$ and $N = \gamma_n(G)$, $Autcent(G)$ is isomorphic to $Hom(G/\gamma_n(G), Z(G))$. Since $\gamma_n(G)=Z(G)$, therefore $Autcent(G)$ is isomorphic to $Hom(G/Z(G), \gamma_{n}(G))$. As $Aut_{c}^{n-1}(G) = Autcent(G)$, it follows that $Aut_{c}^{n-1}(G)$ is isomorphic to $Hom(G/Z(G), \gamma_{n}(G))$.
\end{enumerate}
%
~\\
Conversely, assume that $\gamma_{n}(G)=Z(G)$ and $Aut_{c}^{n-1}(G)$ is isomorphic to $Hom(G/Z(G), \gamma_{n}(G))$. Now as $\gamma_n(G) = Z(G)$, therefore $G$ is purely non-abelian and \\
$Aut_{c}^{n-1}(G)\subseteq Autcent(G)$. Hence, by using Theorem \ref{t1}, we get,
$$
\begin{array}{lcl}
|Autcent(G)|&=&|Hom(G/\gamma_{n}(G), Z(G))| \\[2mm]
&=& |Hom(G/Z(G),\gamma_{n}(G))| \\[2mm]
&=&|Aut_{c}^{n-1}(G)|.
\end{array}
$$

Therefore, $Autcent(G)=Aut_{c}^{n-1}(G)$. Clearly,   $exp(G/\gamma_2(G)) > var(\gamma_{n}(G),Z(G))$.
\end{proof}

Now we prove  the result of Kalra and Gumber \cite[Theorem 3.1]{Kalra}.
\begin{corollary}\cite[Theorem 3.1]{Kalra}\label{c3}
Let $G$ be a finite non-abelian $p$-group. Then $Aut_{c}(G) = Autcent(G)$ if and only if $Aut_{c}(G)$ is isomorphic to $ Hom(G/Z(G),\gamma_{2}(G))$ and $\gamma_{2}(G) = Z(G)$.
\end{corollary}

\begin{proof}
Note that $var(\gamma_{2}(G),Z(G)) < exp(\gamma_{2}(G)) =  exp(G/Z(G)) \leq exp(G/\gamma_{2}(G))$. Now the result follows
from Theorem \ref{t4}.
\end{proof}

\begin{lemma}\label{l2}
Let $n\geq 2$ and let $G$ be a finite non-abelian $p$-group such that $Aut_{Z(G)}^{\gamma_{n}(G)}(G) = Autcent(G)$. Then $G$ is purely non-abelian and $\gamma_{n}(G)\leq Z(G)$.
\end{lemma}

\begin{proof}
Let us assume that $Aut_{Z(G)}^{\gamma_{n}(G)}(G) = Autcent(G)$. Then by the similar argument as in Theorem \ref{t4}, we get $\gamma_n(G) \leq Z(G)$.
Next, we show that $G$ is purely non-abelian. Suppose $G$ is not purely non-abelian, say $ G = H\times A$ where $H$ is a purely non-abelian subgroup and $A$ is a non-trivial abelian direct factor of $G$. Note that $\gamma_n(H) \leq Z(H)$. Then by using Lemma \ref{l6} and Theorem \ref{c1}, we get $|Autcent(G)|>|Autcent(H)|\geq |Aut_{Z(H)}^{\gamma_{n}(H)}(H)|=|Aut_{Z(G)}^{\gamma_{n}(G)}(G)|=|Autcent(G)|$, which is absurd. Hence $G$ is purely non-abelian.
\end{proof}

The following theorem is a particular case of \cite[Corollary C]{Maly}. Here we have provided a different proof.

\begin{theorem}\label{t5}
Let $G$ be a finite non-abelian $p$-group and let $n\geq 2$. Let $Aut_{Z(G)}^{\gamma_{n}(G)}(G) = Autcent(G)$. Then
\begin{enumerate}[$(1)$]
\item $\gamma_{n}(G)\leq Z(G)$ and $d(\gamma_n(G)) = d(Z(G))$.
\item If $exp(G/\gamma_2(G))>var(\gamma_{n}(G),Z(G))$, then  $\gamma_{n}(G) = Z(G)$.
\end{enumerate}
Conversely, if $\gamma_{n}(G)=Z(G)$, then $exp(G/\gamma_2(G))>var(\gamma_{n}(G),Z(G))$ and
$Aut_{Z(G)}^{\gamma_{n}(G)}(G) = Autcent(G)$
\end{theorem}

\begin{proof}
\begin{enumerate}
\item Let us assume that $Aut_{Z(G)}^{\gamma_{n}(G)}(G) = Autcent(G)$. Then by Lemma \ref{l2}, $G$ is purely non-abelian and $\gamma_{n}(G)\leq Z(G)$. As $G$ is purely non-abelian, by Theorem \ref{t1}, we get $|Autcent(G)| =|Hom(G/\gamma_{2}(G), Z(G))|$.

Now for $f \in Aut_{Z(G)}^{\gamma_{n}(G)}(G)\subseteq Aut^{\gamma_{n}(G)}(G) $,  by Theorem \ref{t6}(1), if we take $M = \gamma_n(G)$ and $N=\{1\}$, the map $\alpha_f \in Hom(G, \gamma_{n}(G))$. As $ \gamma_2(G) \subseteq ker(\alpha_f)$; $~~~\alpha_f$ can be extended to a homomorphism from $G/\gamma_2(G)$ to $\gamma_{n}(G)$. Therefore, $|Aut_{Z(G)}^{\gamma_{n}(G)}(G)|\leq |Hom(G/\gamma_2(G),\gamma_{n}(G))|$.
Thus we have,
$$
\begin{array}{lcl}
|Hom(G/\gamma_2(G),Z(G)) | &=& | Autcent(G)| \\[2mm]
                                             &=&|Aut_{Z(G)}^{\gamma_{n}(G)}(G)| \\[2mm]
                                             &\leq& |Hom(G/\gamma_2(G),\gamma_{n}(G))|.

\end{array}
$$

Hence, $ Hom(G/\gamma_2(G),Z(G)) = Hom(G/\gamma_2(G),\gamma_{n}(G)).$ Therefore by using Lemma \ref{l1}, $d(\gamma_{n}(G))= d(Z(G))$.\\

\item Since $exp(G/\gamma_2(G))> var(\gamma_{n}(G),Z(G))$, therefore by Lemma \ref{l1}, we get $\gamma_{n}(G)=Z(G)$.
\end{enumerate}

~\\
Conversely, assume that $\gamma_{n}(G)=Z(G)$. Then we have
$$Autcent(G) = Aut^{Z(G)}_{\gamma_{2}(G)}(G)\subseteq Aut^{Z(G)}_{\gamma_{n}(G)}(G) = Aut_{Z(G)}^{\gamma_{n}(G)}(G).$$
Thus $Autcent(G)\subseteq Aut_{Z(G)}^{\gamma_{n}(G)}(G)$. Hence $Aut_{Z(G)}^{\gamma_{n}(G)}(G)=Autcent(G)$.

Since $\gamma_n(G) = Z(G)$,  $exp(G/\gamma_2(G))> var(\gamma_{n}(G),Z(G))$.
\end{proof}

The result of  Rai \cite[Theorem B(1)]{Rai} follows from the following corollary.
\begin{corollary}\label{c4}
Let $G$ be a finite non-abelian $p$-group. Then $Aut_{Z(G)}^{\gamma_{2}(G)}(G) = Autcent(G)$ if and only if $\gamma_{2}(G) = Z(G)$.
\end{corollary}

\begin{proof}
Note that $var(\gamma_{2}(G),Z(G)) < exp(\gamma_{2}(G)) =  exp(G/Z(G)) \leq exp(G/\gamma_{2}(G))$. Now  the proof follows
from Theorem \ref{t5}.

\end{proof}

\vspace{0.5cm} \noindent {\bf ACKNOWLEDGEMENT}

I would like to thank the referee for his/her valuable comments, specially for suggesting Lemma 6 and Theorem 3.3 for better presentation of the article.

\end{document}